\documentclass[a4paper]{article}
 \usepackage[usenames]{color}
\usepackage{amsthm}
\usepackage{amssymb}
\usepackage{latexsym}
\usepackage{euscript}
\usepackage{amsmath}
\usepackage{amstext}
\usepackage{amsgen}
\usepackage{amsbsy}
\usepackage{amsopn}
\usepackage{amsfonts}
\usepackage[sans]{dsfont}
\usepackage[all]{xy}
\usepackage{hyperref}
\usepackage{graphicx}

\title{Transversal Homotopy Theory}
\author{Jon Woolf}
\date{October 2009}

\newtheorem{theorem}{Theorem}[section]
\newtheorem{proposition}[theorem]{Proposition}
\newtheorem{corollary}[theorem]{Corollary}
\newtheorem{lemma}[theorem]{Lemma}

\theoremstyle{definition}

\newtheorem{remark}[theorem]{Remark}
\newtheorem{remarks}[theorem]{Remarks}

\newcommand{\defn}[1]{\emph{#1}}

\newcommand{\ie}{i.e.\ }

\newcommand{\bpt}{{\star}}

\newcommand{\N}{\mathbb{N}}
\newcommand{\Z}{\mathbb{Z}}

\newcommand{\R}{\mathbb{R}}

\renewcommand{\P}{\mathbb{P}}

\newcommand{\cat}[1]{\mathbf{#1}}
\newcommand{\sns}{\cat{Whit}}
\newcommand{\bsns}{\cat{Whit}_\bpt}
\newcommand{\mfld}{\cat{Mfld}}
\newcommand{\bmfld}{\cat{Mfld}_\bpt}
\newcommand{\grp}{\cat{Grp}}
\newcommand{\mon}{\cat{Mon}}

\newcommand{\id}{id}

\newcommand{\colim}{\mathrm{colim}\, }

\newcommand{\psidot}[2]{\psi_{#1}\left({#2}\right)} 
\newcommand{\Psidot}[2]{{\Psi}_{#1}\left({#2}\right)}

\newcommand{\tang}[4]{ {}^{#1}_{#2} \cat{Tang}^{#3}_{#4} } 
\newcommand{\frtang}[3]{ \tang{ {\rm fr} }{#1}{#2}{#3} }

\newcommand{\trans}[2]{ \left[ #1, #2\right]_+} 

\newcommand{\fatwedge}{\mbox{\footnotesize $\mathbb{V}$}}
\newcommand{\fatsusp}{\mathbb{S}}
\newcommand{\fatthom}{\mathbb{TH}}

\begin{document}

\maketitle

\section{Introduction}

This paper sets out the beginnings of a `transversal homotopy theory' of smooth stratified spaces. The idea is to mimic the constructions of homotopy theory but using only those smooth maps which are transversal to each stratum of some fixed stratification of the space. This lovely idea is due to John Baez and James Dolan; I learnt about it from the discussion \cite{Baez:kx}, which not only lays out the basic idea but also the key examples. For an informal introduction containing all the intuition, it is well worth looking at this discussion. The aim of this paper is to provide the technical backbone for the ideas expressed there. 

The essence of the theory can be understood by considering this example: Suppose a path in a manifold crosses a codimension one stratum $S$ transversally, and then turns around and crosses back. Clearly, this path is homotopic to one which does not cross $S$ at all. However, at some point in the homotopy --- at the point where the path is `pulled off' $S$ --- we must pass through a path which is not transversal to $S$. Therefore, if we insist that the equivalence on paths is homotopy through transversal paths, we obtain a theory in which crossings cannot be cancelled --- the class of a path now remembers its interactions with $S$. It is very important that the notion of homotopy used here is that of a family of transversal maps and {\em not} that of a transversal family of maps. In the latter case we would recover the usual homotopy theory (because any continuous map of Whitney stratified spaces can be approximated by a smooth transversal map). 

Before giving any further details we sketch the `big picture'. A topological space $X$ has an associated {\em fundamental $\infty$-groupoid} $\Pi X$. One model for $\Pi X$ is the singular simplicial set of $X$ (which is a Kan complex and so can be viewed as an $\infty$-groupoid), but we use a cubical model instead, in which $d$-morphisms are maps $[0,1]^d \to X$. In order to obtain more tractable invariants one restricts attention to a range $k\leq d \leq n+k$ of dimensions. For this to make good sense one imposes boundary conditions --- the first $k$ boundaries of each map should map to a chosen basepoint in $X$ --- and considers the maps $[0,1]^{n+k}\to X$ up to homotopy. The resulting invariant $\Pi_{k,n+k}X$ is expected to be a `$k$-tuply monoidal $n$-groupoid'. The simplest and most familiar case is when $n=0$ and we obtain the homotopy groups $\Pi_{k,k}X=\pi_kX$. The next simplest is the fundamental groupoid, $\Pi_{0,1}X$ in this notation, and its higher dimensional analogues. 

The homotopy theory of sufficiently nice spaces, e.g. CW complexes, is equivalent to the theory of $\infty$-groupoids.  Do other categorical structures have an analogous spatial interpretation? Weakening the groupoid condition a little one obtains (higher) categories with duals. Two key examples of this weaker structure are the category of Hilbert spaces, in which each morphism has an adjoint, and cobordism categories, in which the dual is provided by `turning a cobordism upside down'. Baez and Dolan's insight is that transversal homotopy theory should associate categories with duals to stratified spaces just as homotopy theory associates groupoids to spaces. Much of this paper is devoted to showing that this is true, at least for the analogues of $\Pi_{k,k}X$ and $\Pi_{k,k+1}X$,  when $X$ is a Whitney stratified manifold. However, it is important to note that not every category with duals arises in this way; to obtain a correspondence one would require considerably more general objects on the geometric side.  

In slightly more detail the contents are as follows. \S\ref{whitney} and \ref{transversality} provide a resum\'e of Whitney stratified manifolds and transversal maps between them. The only novelty is the notion of a stratified normal submersion. These maps preserve transversality under post-composition. The category $\sns$ of Whitney stratified manifolds and stratified normal submersions contains the category $\mfld$ of smooth manifolds and smooth maps as a full-subcategory. 

Some basic constructions within $\sns$, and its based cousin $\bsns$, are discussed in \S\ref{category whit}, namely fibred products, coproducts, suspensions and Thom spaces. To be more precise, coproducts, suspensions and Thom spaces only exist up to homotopy in $\bsns$. This is because the usual constructions produce non-manifolds, and so we make a unique-up-to-homotopy  choice of `fattening' of these to obtain objects of $\bsns$. Alternatively, one could enlarge the category to include more general stratified spaces, but it turns out that the `fattened' versions behave better with respect to transversal homotopy theory. The missing construction is that of mapping spaces, in particular loop spaces. To have these, one would need to enlarge $\bsns$ to include `$\infty$-dimensional Whitney stratified manifolds'. Section \ref{htpy thru transversal} contains some brief but important remarks and results about homotopies through transversal maps. 

In \S \ref{monoids}  the `transversal homotopy monoids' of a Whitney stratified space are defined. For $n>0$ the $n$th transversal homotopy monoid is a functor $$\psidot{n}{-} : \bsns \to \mon$$  valued in the category of dagger monoids. When $n=0$ there is a similar functor valued in the category of pointed sets. These functors generalise the usual homotopy groups in the sense that there are commutative diagrams
$$
\xymatrix{
\bmfld \ar@{_{(}->}[d] \ar[r]^{\pi_n} & \grp \ar@{_{(}->}[d] \\
\bsns  \ar[r]_{\psi_n} & \mon.
}
$$
We use $\psi$ because it is reminiscent of the symbol for a transversal intersection.

In \S\ref{ex1} we consider the example of transversal homotopy monoids of the space $\mathbb{S}^k$ obtained by stratifiying $S^k$ by a point and its complement. These play a central r\^ole because the transversal homotopy monoids of wedges of spheres can be organised into an operad for which all transversal homotopy monoids are algebras. The Pontrjagin--Thom construction provides a geometric interpretation of $\psidot{n}{\mathbb{S}^k}$ as the codimension $k$ framed tangles in $n$ dimensions, \ie the set of ambient isotopy classes of codimension $k$ submanifolds of $\R^n$. 

We briefly discuss the behaviour of transversal homotopy monoids under products in \S\ref{products}. In  \S\ref{first monoid} we gather together some observations and results about the first homotopy monoid $\psidot{1}{X}$. This has a combinatorial description, due to Alexey Gorinov, in terms of loops in a certain labelled graph. 

In \S \ref{transversal homotopy categories} we discuss the analogues of the fundamental groupoid and the higher homotopy groupoids. These are functors which assign a category $\Psidot{n,n+1}{X}$ to each Whitney stratified manifold $X$. As an example, we show in \S\ref{ex2} that $\Psidot{n,n+1}{\mathbb{S}^k}$ is the category of  category of codimension $k$ framed tangles in dimensions $n$ and $n+1$, \ie the category consisting of closed submanifolds of $\R^n$ and (ambient isotopy classes of) bordisms in $\R^{n+1}$ between them. These `transversal homotopy categories' have a rich structure: they are rigid monoidal dagger categories  which are braided, in fact ribbon, for $n\geq 2$ and symmetric for $n\geq 3$. This structure arises by considering them as `algebras' for the transversal homotopy categories of wedges of spheres (see Theorem \ref{structure thm}).

In \S \ref{thom} we sketch out the generalisation from spheres to other Thom spectra. In \S\ref{tangle hypothesis} we briefly discuss the relation of transversal homotopy theory to the Tangle Hypothesis of Baez and Dolan.  

Appendix \ref{collapse maps} contains some technical details about `collapse maps', which are key to  the Pontrjagin--Thom construction. This material is well-known, but we add a few refinements necessary for our setting. 

\subsection*{Acknowledgments}

I am in great debt to John Baez, and to all those who contributed to the discussion on the $n$-category caf\'e, for the idea of studying transversal homotopy theory. I would like to thank Alexey Gorinov and Conor Smyth for many helpful discussions. This work is funded, very generously, by the Leverhulme Trust (grant reference: F/00 025/AH).

\section{Preliminaries}
\label{prelim}

\subsection{Whitney stratified spaces}
\label{whitney}

A \defn{stratification} of a smooth manifold $X$ is a  decomposition $X = \bigcup_{i\in \mathcal{S}} S_i$ into disjoint subsets $S_i$ indexed by a partially-ordered set $\mathcal{S}$ such that
\begin{enumerate}
\item the decomposition is locally-finite,
\item $S_i \cap \overline{S_j} \neq \emptyset \iff S_i \subset \overline{S_j}$, and this occurs precisely when $i \leq j$ in $\mathcal{S}$,
\item each $S_i$ is a locally-closed smooth connected submanifold of $X$.
\end{enumerate}
The $S_i$ are referred to as the \defn{strata} and the partially-ordered set $\mathcal{S}$ as the \defn{poset of strata}. The third condition is usually called the \defn{frontier condition}. 

Nothing has been said about how the strata fit together from the point of view of smooth geometry. In order to obtain a class of stratified spaces with which we can do differential geometry we need to impose further conditions, proposed by  Whitney \cite{MR0188486} following earlier ideas of Thom \cite{MR0239613}. Suppose $x\in S_i \subset \overline{S_j}$ and that we have sequences $(x_k)$ in $S_i$ and $(y_k)$ in $S_j$ converging to $x$. Furthermore, suppose that the secant lines $\overline{x_ky_k}$ converge to a line $L\leq  T_xX$ and the tangent planes $T_{y_k}S_j$ converge to a plane $P \leq T_xX$. (An intrinsic definition of the limit of secant lines can be obtained by taking the limit of $(x_k,y_k)$ in the blow-up of $X^2$ along the diagonal, see \cite[\S4]{Mather:1970fk}. The limit of tangent planes is defined in the Grassmannian $Gr_d(TX)$ where $d=\dim S_j$. The limiting plane $P$ is referred to as a generalised tangent space at $x$.) In this situation we require
 \begin{description}
\item[(Whitney A)] the tangent plane $T_xS_i$ is a subspace of the limiting plane $P$;
\item[(Whitney B)] the limiting secant line $L$ is a subspace of the limiting plane $P$.
\end{description}
Mather \cite[Proposition 2.4]{Mather:1970fk} showed that the second Whitney condition implies the first. It remains useful to state both because the first is often what one uses in applications, but the second is necessary  to ensure that the normal structure to a stratum is locally topologically trivial, see for example \cite[1.4]{GoM2}.

A \defn{Whitney stratified manifold} is a manifold with a stratification satisfying the Whitney B condition. A 
\defn{Whitney stratified space} $W\subset X$ is a closed union $W$ of strata in a Whitney stratified manifold $X$.  Examples abound, for instance any manifold with the \defn{trivial stratification} which has only one stratum is a Whitney stratified manifold. More interestingly, any complex analytic variety admits a Whitney stratification \cite{MR0188486}, indeed any (real or complex) subanalytic set of an analytic manifold admits a Whitney stratification \cite{MR0377101,H}.

\subsection{Transversality}
\label{transversality}

A smooth map $f:M \to Y$ from a manifold $M$ to a Whitney stratified manifold $Y$ is \defn{transversal} if for each $p\in M$ the composite
$$
T_pM \stackrel{df}{\longrightarrow} T_{fp}Y \to T_{fp}Y / T_{fp}B = N_{fp}B
$$
is surjective. Here $B$ is the stratum of $Y$ containing $f(p)$ and $NB$ is the normal bundle of $B$ in $Y$. More generally, a smooth map $f:X \to Y$ of Whitney stratified manifolds is \defn{transversal} if the restriction of $f$ to each stratum of $X$ is transversal in the above sense. 

\begin{remark}
If $f:X \to Y$ is not transversal then it cannot be made so by refining either the stratifications of $X$ or $Y$ or both. Any map to a manifold with the trivial stratification is transversal.
\end{remark}

We equip the space $C^\infty(X,Y)$ of smooth maps from a manifold $X$ to a manifold $Y$ with the Whitney topology which has basis of open sets given by the subsets
$$
\{ f\in C^\infty(X,Y) \ | \ j^kf(X) \subset U \} \qquad 1\leq k < \infty
$$
for an open subset $U$ of the bundle $J^k(X,Y) \to X \times Y$ of $k$-jets. 
\begin{theorem} \cite{MR520929}
The set of transversal maps is open and dense in $C^\infty(X,Y)$.
\end{theorem}
\noindent The proof of this result uses the Whitney A condition in an essential way. In fact, the set of transversal maps is open if, and only if, the stratifications of $X$ and $Y$ satisfy Whitney A \cite{MR520929}. The proof that it is dense is a corollary of the following result and Sard's theorem, see for example \cite[\S1.3]{GoM2}.

\begin{theorem}
\label{transversality theorem}
Let $X$ and $Y$ be Whitney stratified manifolds, and $P$ a manifold. Suppose $f:X\times P \to Y$ is a transversal map with respect to the product stratification of $X\times P$ and given stratification of $Y$. Then the map $f_p = f|_{X \times\{p\}}$ is transversal for $p\in P$  if and only if $p$ is a regular value of the composite$$f^{-1}(B) \hookrightarrow X \times P \stackrel{\pi}{\longrightarrow} P$$ for each stratum $B$ of $Y$.
\end{theorem}
\begin{proof}
This is a standard result; the crux of the argument is that $(x,p)$ is a regular point if and only if
$$
T_{(x,p)}f^{-1}B + T_{(x,p)}(X\times \{p\}) = T_{(x,p)}(X \times P).
$$
The details can be found, for example, in \cite[Chapter 2]{MR0348781}.
\end{proof}

\begin{proposition}\cite[\S1.3]{GoM2}
\label{whitney transversality}
If $f: X \to Y$ is transversal then the decomposition of $X$ by subsets $A \cap f^{-1}(B)$ where $A$ and $B$ are strata of $X$ and $Y$ respectively is a Whitney stratification. We refer to this stratification as the \defn{stratification induced by $f$} and, when we wish to emphasise it, denote $X$ with this refined stratification by $X_f$.
\end{proposition}

One easy consequence is that the product $X \times Y$ of Whitney stratified manifolds equipped with the \defn{product stratification}, whose strata are the products $A \times B$ of strata $A\subset X$ and $B\subset Y$, is a Whitney stratified manifold. Unless otherwise stated the product will always be equipped with this stratification. 

The composite of transversal maps need not be transversal (for example consider $\R \hookrightarrow \R^2 \hookrightarrow \R^3$ where $\R$ and $\R^2$ are trivially stratified and $\R^3$ has a one-dimensional stratum intersecting the image of $\R$). We now identify a class of maps which preserve transversality under composition, and which are themselves closed under composition.

A smooth map $g:X\to Y$ of Whitney stratified manifolds is \defn{stratified} if for any stratum $B$ of $Y$ the inverse image $g^{-1}B$ is a union of strata of $X$. We say it is a \defn{stratified submersion} if for each stratum $B$ of $Y$ and stratum $A \subset g^{-1}B$ the restriction $g|_A : A \to B$ is a submersion. Alternatively, we say it is a \defn{stratified normal submersion} if the induced map $N_xA \to N_{gx}B$ of normal spaces is surjective. 
\begin{lemma}
\label{gen comp}
Suppose $X,Y$ and $Z$ are Whitney stratified manifolds and $f:X\to Y$ and $g:Y\to Z$ are smooth. If $f:X \to Y$ is a stratified submersion then the composite $g\circ f : X \to Z$ is transversal whenever $g:Y \to Z$ is transversal. If $g:Y \to Z$ is a stratified normal submersion then the composite $g\circ f : X \to Z$ is transversal whenever $f:X \to Y$ is transversal.
\end{lemma}
\begin{proof}
Take a point $x$ in a stratum $A$ of $X$. Let $B$ be the stratum of $Y$ containing $f(x)$ and $C$ the stratum of $Z$ containing $g\circ f (x)$. Consider the diagram:
\[
\xymatrix{
T_xA \ar[r]\ar@{-->}[d]_{\alpha} & T_{x}X \ar[d]_{df}  & \\
T_{fx} B \ar[r]  & T_{fx}Y \ar[d]^{dg} &  \\
 & T_{gfx}Z  \ar[r]& N_{gfx}C.
}
\]
When $f$ is a stratified submersion there is a unique surjection $\alpha$ making the top left square commute. If $g$ is transversal the composite $T_{fx}B \to N_{gfx}C$ is surjective. Hence so is that $T_{x}A \to N_{gfx}C$ and so $g\circ f$ is transversal too.

The proof of the second part is similar.
\end{proof}
There is a partial converse: if $f$ is stratified and $g\circ f$ is transversal whenever $g$ is transversal then $f$ is necessarily a stratified submersion. To see this, take $g$ to be  the identity map on $Y$ but where the target is stratified by a small normal disk to the stratum at $f(x)$ and its complement (the normal disk has two strata, the interior and its boundary). This $g$ is transversal and the condition that $g\circ f$ is transversal at $x$ is equivalent to the statement that $f$ is a stratified submersion at $x$. Similarly if $g$ is stratified and $g\circ f$ is transversal whenever $f$ is transversal then $g$ is necessarily a stratified normal submersion. In this case, take $f$ to be the inclusion of a small normal disk to a stratum at the point $y\in Y$. This is transversal and the fact that $g\circ f$ is transversal at $y$ is equivalent to the statement that $g$ is a stratified normal submersion at $y$. 

\begin{lemma}
If $f:X \to Y$ is a transversal map of Whitney stratified manifolds then $f$ is a stratified normal submersion with respect to the stratification of $X$ induced by $f$ and the given stratification of $Y$. 
Conversely, a stratified normal submersion $f:X\to Y$ becomes a transversal map if we forget the stratification of $X$, \ie give $X$ the trivial stratification. 
\end{lemma}
\begin{proof}
By construction $f$ is stratified after we refine the stratification of $X$ so that the strata are of the form $A\cap f^{-1}B$ where $A$ and $B$ are strata of $X$ and $Y$ respectively. Suppose $p \in A \cap f^{-1}B$. Since $f$ is transversal $df: T_pA \to N_{f(p)}B$ is surjective, and since $T_p(A \cap f^{-1}B)$ is in the kernel, $N_p(A \cap f^{-1}B) \to N_{f(p)} B$
is also surjective. The second statement is clear. 
\end{proof}

\subsection{A category of Whitney stratified manifolds}
\label{category whit}

The identity map of a Whitney stratified manifold is a stratified normal submersion, and the composite of two stratified normal submersions is a stratified normal submersion. It follows that there is a category $\sns$ whose objects are Whitney stratified manifolds and whose morphisms are stratified normal submersions.  When $Y$ is trivially stratified with only one stratum then any smooth map $X \to Y$ is a stratified normal submersion. So $\sns$ contains the category of smooth manifolds and smooth maps as a full subcategory. (There is also a category of Whitney stratified manifolds and stratified submersions, but this does not contain the category of manifolds and smooth maps as a full subcategory.)

There are evident notions of homotopy and homotopy equivalence in $\sns$ given by the usual definitions with the additional requirement that all maps should be stratified normal submersions. For example a homotopy is a stratified normal submersion $X \times [0,1] \stackrel{h}{\longrightarrow} Y$
with the property that each slice $h_t: X \times \{t\} \to Y$ is also a stratified normal submersion and so on.

In the remainder of this section we describe some basic constructions in the category $\sns$, and also in the based analogue $\bsns$. The basepoint is always assumed to be generic, \ie it lies in an open stratum, equivalently the inclusion map is transversal. 

\subsubsection*{Fibred products}

Both $\sns$ and $\bsns$ have all finite products --- these are given by the Cartesian product of the underlying manifolds equipped with the product stratification. We also have fibred products 
$$
\xymatrix{
X\times_Z Y \ar[r] \ar[d] & Y \ar[d]^g \\
X \ar[r]_f & Z
}
$$
whenever $f$ and $g$ are {\em transversal} to one another in $\sns$ (or in $\bsns$). By this we mean that for every pair  $(A,B)$ of strata of $X$ and $Y$ respectively, the restrictions $f|_A:A \to Z$ and $g|_B:B \to Z$ are transversal in the usual sense of smooth maps of manifolds. This is equivalent to requiring that $f\times g: X \times Y \to Z^2$ is transversal with respect to the stratification of $Z^2$ by the diagonal $\Delta$ and its complement, \ie that 
$$
d(f\times g)T(A \times B) +T\Delta = TZ^2
$$
for every $A$ and $B$. In this situation Proposition \ref{whitney transversality} shows that $(f\times g)^{-1}\Delta$ is a Whitney stratified submanifold of $X \times Y$. The fibre product is defined to be
$$
X\times_Z Y = (f\times g)^{-1}\Delta
$$
(with the above stratification) and the maps to $X$ and $Y$ are given by the inclusion into $X\times Y$ followed by the projections. The inclusion is a stratified normal submersion and therefore so are the maps to $X$ and $Y$.

Note that the stratification of $Z$ does not explicitly appear in this discussion. The previous notion of a transversal map $f:X \to Y$ of Whitney stratified manifolds corresponds to the special case when $f$ is transversal (in the above sense) to the map $g:Y\to Z$ which forgets the stratification of $Y$, \ie $Z$ is the underlying manifold of $Y$ with the trivial stratification and $g$ the identity. The fibre product in this case is $X$ equipped with the induced stratification.

\subsubsection*{Coproducts}

The category $\sns$  has finite coproducts given by disjoint union. However $\bsns$ does not have coproducts because the wedge sum $X \vee Y$ is not in general a manifold. To avoid this difficulty we could enlarge the category so that it contained all Whitney stratified spaces. However, there would still be a problem in that the basepoint of $X \vee Y$ is not generic, and the inclusions of $X$ and $Y$ in $X \vee Y$ do not satisfy any reasonable extension of the notion of stratified normal submersion at this point. This is one of several similar situations in which it seems better, for the purposes of transversal homotopy theory, to modify a construction so that we remain within $\bsns$. 
\begin{lemma}
\label{fattening}
Suppose $W$ is a Whitney stratified space with only isolated singularities. Let $S$ be the set of these singularities and further suppose that  there is an open neighbourhood $U$ of $S$ such that $U-S$ is contained in the union of the open strata of $W$. Then we can choose a `fattening'  of $W$ which is a Whitney stratified manifold $\widetilde{W}$ with smooth maps 
$$
(W,U) \stackrel{\imath}{\longrightarrow} (\widetilde{W}, p^{-1}U)  \stackrel{p}{\longrightarrow} (W,U)
$$
such that there is a smooth homotopy $p\circ \imath \simeq \id_W$ relative to $W-U$, and a smooth homotopy $\imath\circ p \simeq \id_{\widetilde{W}}$ in $\sns$. The construction of such a fattening depends on  the ambient manifold in which $W$ is embedded and several other choices, but the resulting Whitney stratified manifold $\widetilde{W}$  is unique up to homotopy equivalence in $\sns$.
\end{lemma}
\begin{proof}
Suppose $M\supset W$ is (a choice of) ambient manifold for $W$. Let $\pi:N \to W-S$ be a tubular neighbourhood of $W-S$ in $M$ and let $B$ be the union of disjoint open balls about each singularity such that $B\cap W \subset U$. Let $\widetilde{W} = N \cup B$ stratified by the pre-images under $\pi$ of the strictly positive codimension strata in $W-S$ and (the connected components of) the union of $B$ with the pre-image of the open strata. 

Let $\imath:W \to \widetilde{W}$ be the inclusion. To define $p$, choose inward-pointing (in the normal direction) radial vector fields on $N$ and on $B$ and patch them using a partition of unity. Rescale so that the flow at time $1$ smoothly maps $B\cup N$ onto the subspace $W$, and let this map be $p$. These maps satisfy the stated properties. If $(\widetilde{W}, \imath,p)$ and $(\widetilde{W}', \imath',p')$ are two choices of fattening, then the maps
$$
\widetilde{W} \stackrel{\imath'\circ p}{\longrightarrow} \widetilde{W}' \stackrel{\imath\circ p'}{\longrightarrow} \widetilde{W}
$$
give a homotopy equivalence in $\sns$. 
\end{proof}

Define the `fat wedge' of Whitney stratified manifolds $X$ and $Y$ in $\bsns$ to be the fattening 
$$
X \fatwedge Y = \widetilde{X \vee Y}
$$
of the wedge product (with the usual basepoint, which becomes generic in the fattening). The conditions of the lemma are satisfied because the basepoints of $X$ and $Y$ are generic. See Figure \ref{fat circle wedge} for an example. Given maps $f:X \to Z$ and $g:Y\to Z$ in $\bsns$, the diagram
$$
\xymatrix{
& Y \ar@{_{(}->}[d] \ar[ddr]^g & \\
X  \ar@{^{(}->}[r] \ar[drr]_f &
X \fatwedge  Y \ar@{-->}[dr] & \\
&& Z
}
$$
commutes up to homotopy in $\bsns$ where the dotted arrow is the composite 
$$
X \fatwedge Y \stackrel{p}{\longrightarrow} X \vee Y \stackrel{f\vee g}{\longrightarrow} Z.
$$  
If we fix a choice of fattening $X \fatwedge Y$ for each $X$ and $Y$ then we can define
$$
f\fatwedge g = \imath \circ (f\vee g) \circ p
$$
for maps $f:X \to X'$ and $g:Y \to Y'$. This is a stratified normal submersion, and we obtain a functor taking values in the homotopy category of $\bsns$ (whose objects are the same as those of $\bsns$, but whose morphisms are homotopy classes of maps in $\bsns$). Different choices of fattening lead to naturally isomorphic functors.

\begin{figure}
\centerline {
\includegraphics[width=3in]{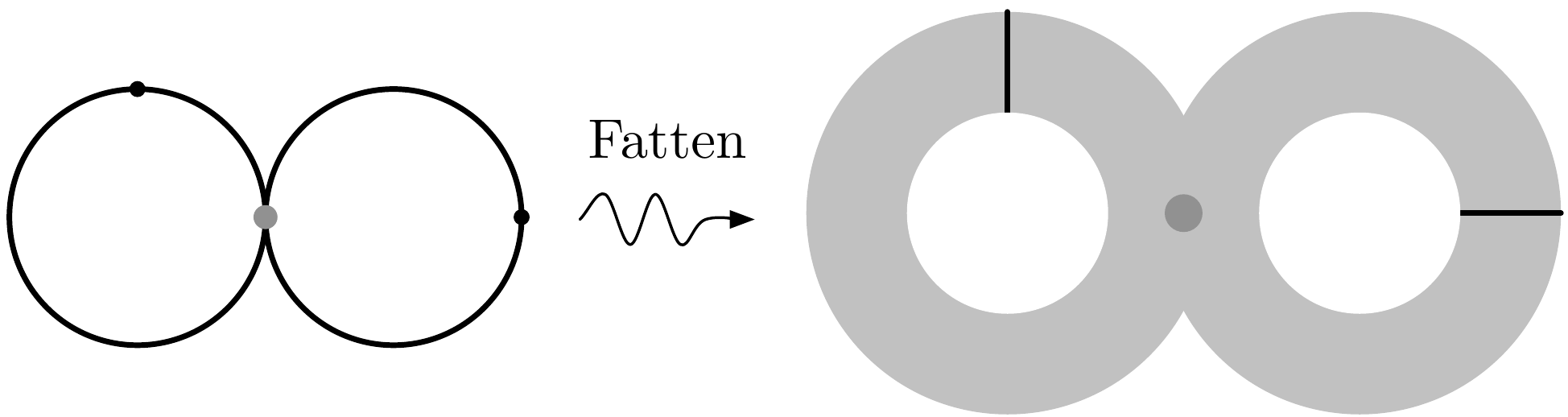}
}
\caption{The wedge (on the left) and the fat wedge (on the right) of two circles stratified by a point and its complement. The grey dots represent basepoints, the black dots and lines represent strata. We could equally well embed $S^1\vee S^1$ in $\R^3$ and obtain a solid double torus, stratified by two disks and their complement, as the fattening (and similarly for higher dimensional $\R^n$).}
\label{fat circle wedge}
\end{figure}

\subsubsection*{Suspensions}

A similar approach can be used to define suspensions. Suppose $X$ is a Whitney stratified manifold. Let $SX = X\times [-1,1]/ X \times\{\pm 1\}$ be the usual suspension. Stratify it by the positive codimension strata of $X$, thought of as the subspace $X\times\{0\}$, the two suspension points (which are singular in general) and the remainder. This is a Whitney stratified space satisfying the conditions of the lemma, so we can define the `fat suspension' 
$$
\fatsusp X = \widetilde{S X}.
$$ 
This is compatible with our notation for the sphere stratified by a point and its complement in the sense that $\fatsusp \mathbb{S}^n = \mathbb{S}^{n+1}$. (More precisely they are homotopy equivalent, but since the fat suspension is only defined up to homotopy equivalence in $\bsns$ we are free to choose the representative.)

The construction is functorial: if $f: X \to Y$ is a map in $\sns$ then define 
$$
\fatsusp f = \imath \circ Sf \circ p : \fatsusp X \to \fatsusp Y.
$$
This yields a suspension functor valued in the homotopy category, which is well-defined up to natural isomorphism. There is also a based version for $\bsns$ in which the basepoint of $\fatsusp X$ is the old basepoint, thought of as lying in $X\times\{0\}$. 

\subsubsection*{Thom spaces}

Let  $E \to B$ be a real vector bundle with structure group $O(n)$, with unit disk bundle $D(E)$ and sphere bundle $S(E)$. We stratify the Thom space
$$
{\rm TH}(E) = D(E)/S(E)
$$
by the zero section $B$, the singular point `at infinity' and the remainder. This satisfies the conditions of the above lemma and so we define the `fat Thom space' 
$$
\fatthom(E) = \widetilde{ {\rm TH}(E) }
$$
with the point at infinity, which is now generic, as basepoint. If $f: E \to g^*E'$ is a bundle map which is an orthogonal isomorphism on each fibre then
$$
\fatthom(f) = \imath\circ {\rm TH}(f) \circ p: \fatthom(E) \to \fatthom(E')
$$
 is a based stratified normal submersion. This gives a functor from bundles and bundle maps, satisfying the above conditions, to the homotopy category of $\bsns$. To repeat the mantra one last time; different choices of fattening yield naturally isomorphic functors.

\subsection{Homotopy through transversal maps}
\label{htpy thru transversal}

Suppose $X$ is a Whitney stratified manifold. Stratify $X\times [0,1]$ as a product. A \defn{homotopy through transversal maps} to a Whitney stratified manifold $Y$ is a smooth map $h:X\times [0,1] \to Y$ such that each slice
$$
h_t = h( \cdot , t): X \to Y
$$
is transversal. We insist that
\begin{equation}
\label{h boundary conditions}
h_t = \left\{
\begin{array}{ll}
h_0 & t\in [0,\epsilon)\\
h_1 & t\in [1-\epsilon,1]
\end{array}
\right.
\end{equation}
for some $\epsilon >0$. Homotopy through transversal maps is an equivalence relation on transversal maps from $X$ to $Y$. Every homotopy $h:X\times[0,1] \to Y$ through transversal maps is also a transversal map but not vice versa.

The following result is a cornerstone of the theory of stratified spaces, introduced by Thom in \cite{MR0239613} and proved by Mather in the notes \cite{Mather:1970fk}. 
\begin{theorem}[Thom's first isotopy lemma]
\label{TFIL}
Suppose $X$ is a Whitney stratified subset of a manifold $M$ and $f: M \to \R^n$ a smooth map whose restriction to $X$ is proper and a stratified submersion. Then there is a commutative diagram
$$
\xymatrix{ 
X \ar[r]^{h\qquad\qquad} \ar[dr]_f & (f^{-1}(0) \cap X) \ar[d]^{\pi}  \times  \R^n\\
& \R^n
}
$$
in which $h$ is a stratified homeomorphism (\ie a continuous stratified map with continuous stratified inverse) whose restriction to each stratum is smooth. 
\end{theorem}

\begin{proposition}
\label{deformation}
Suppose $X$ is a compact Whitney stratified space and $h:X\times [0,1] \to Y$ a homotopy through transversal maps. Then $X_{h_0}$ can be continuously deformed into $X_{h_1}$, \ie there is a continuous isotopy from the identity on $X$  to a stratified homeomorphism  $X_{h_0} \to X_{h_1}$. 
\end{proposition}
\begin{proof}
For some $\epsilon >0$, we can extend the homotopy to a smooth map
$h:X\times (-\epsilon, 1+\epsilon) \to Y$ which is still transversal on each slice. Theorem \ref{transversality theorem} tells us that the projection $$\pi :  \left(X \times (-\epsilon,1+\epsilon) \right)_h \to (-\epsilon,1+\epsilon)$$ is a proper stratified submersion. Hence by Thom's first isotopy lemma there is a stratified homeomorphism 
$q:\left(X \times (-\epsilon,1+\epsilon)\right)_h\to  X_{h_0} \times (-\epsilon,1+\epsilon)$
such that $\pi\circ q =\pi$. Let $q_t$ denote the restriction  $\pi_X \circ q(-,t) :  X \to X$ of $q$ to the slice labelled by $t$. Then the composite
$$
\xymatrix{
X \times [0,1] \ar[r]^{q_0 \times \id} & X \times [0,1] \ar[r]^{q^{-1}} & X \times [0,1] \ar[r]^{\qquad \pi_X} & X
}
$$
is the desired continuous isotopy from the identity to a stratified homeomorphism $X_{h_0} \to X_{h_1}$. 
\end{proof}
\begin{remark}
\label{smooth deformation}
In special cases $X_{h_0}$ can be {\em smoothly} deformed into $X_{h_1}$. For instance this is so when the stratification of $Y$ has strata $Y_i-Y_{i-1}$ for $i=0,\ldots, k$ where 
$$
\emptyset = Y_{-1} \subset Y_0 \subset \cdots \subset Y_{k-1} \subset Y_k =Y
$$
is a filtration by closed submanifolds. The proof is the same except that this condition ensures that the induced stratification of $X \times (-\epsilon,1+\epsilon)$ and the projection to $(-\epsilon,1+\epsilon)$ satisfy the hypotheses of Proposition \ref{smooth TFIL} below, which we therefore use in place of Thom's first isotopy lemma.
\end{remark}

\begin{proposition}
\label{smooth TFIL}
Let $X$ be a Whitney stratified subset of a manifold $M$ with strata $S_0,\ldots,S_k$ such that $\overline{S_i} = S_0\cup  \cdots\cup S_i$ is a smooth submanifold of $M$. Suppose $f:M \to \R$ is a smooth map whose restriction to $X$ is a proper stratified submersion. Then there is a commutative diagram
$$
\xymatrix{ 
X \ar[r]^{h\qquad\qquad} \ar[dr]_f & (f^{-1}(0) \cap X) \ar[d]^{\pi}  \times  \R\\
& \R
}
$$
in which $h$ is a stratified diffeomorphism. 
\end{proposition}
\begin{proof}
It is sufficient to construct a smooth controlled lift $V$ of $\partial / \partial x$, \ie a smooth vector field on $X$ which is tangential to the strata and such that $f_*V=\partial / \partial x$. We construct this inductively on the $\overline{S_i}$. Choose a metric on $M$ and let $W_i=  \nabla ( f|_{S_i} )$ be the gradient of $f|_{S_i}$ with respect to the restriction of the metric to $S_i$. Note that $W_i \cdot f >0$ because $f$ is a stratified submersion. 

Assume, inductively, that we have constructed a smooth controlled lift $V_{i-1}$ on the smooth submanifold $\overline{S_{i-1}}$. The base case is provided by setting 
$$
V_0= \frac{W_0}{W_0 \cdot f},
$$
so that $df{V_0} = V_0\cdot f = 1$ as required. Extend $V_{i-1}$ to a vector field $V_{i-1}'$ on an open neighbourhood $U_{i-1}$ of the smooth submanifold $\overline{S_{i-1}}$ of $\overline{S_i}$. We may assume, by restricting to a  smaller neighbourhood if necessary, that $V_{i-1}' \cdot f > 0$. Choose a partition of unity $\{\alpha,\beta\}$ with respect to the cover $\{ U_{i-1}, S_i\}$ of $\overline{S_i}$ and let $V_i'' = \alpha V_{i-1}' +\beta W_i$. Since $V_i'' \cdot f >0$ we can normalise this to obtain $V_i = V_i'' / V_i'' \cdot f$ which is the desired controlled lift on $\overline{S_i}$
\end{proof}

\section{Transversal homotopy monoids}
\label{monoids}

\subsection{Definition}
Let $X$ be a Whitney stratified manifold. In this section all spaces will be equipped with a generic basepoint $\bpt$ and all maps will be based, unless otherwise stated.  For $n\in \N$ we fix a choice of small  disk-shaped closed neighbourhood $B^n \subset S^n$ of the basepoint. Define $\psidot{n}{X}$ to be the set of equivalence classes of smooth transversal maps $(S^n,B^n) \to (X,\bpt)$ under the equivalence relation of homotopy through such maps.  Denote the class of a transversal map $f$ by $[f]$.  

The set $\psidot{0}{X}$ is the set of open strata of $X$. For $n\geq 1$ the set $\psidot{n}{X}$ has the structure of a monoid: define $[f]\cdot [g] = [(f\vee g) \circ \mu]$ where 
$$
\mu : (S^n,B^n) \to (S^n \vee S^n,\bpt)
$$
is a smooth map which is a diffeomorphism when restricted to the inverse image of $(S^n-B^n)+ (S^n - B^n)$. This is associative; the class of the constant map to $\bpt$ is the unit for the operation. We refer to $\psidot{n}{X}$ as the \defn{$n$th transversal homotopy monoid} of $X$. The usual Eckmann--Hilton argument shows that $\psidot{n}{X}$ is a {\em commutative} monoid for $n\geq 2$. 
\begin{remarks}
\begin{enumerate}
\item {\it A priori} the definition of $\psidot{n}{X}$ depends on our choice of $B^n$ and that of the product depends on the choice of $\mu$. In fact, the set $\psidot{n}{X}$ is well-defined up to canonical isomorphism independently of the choice of $B^n$ and the product is independent of the choice of $\mu$. Nevertheless, for technical convenience later we fix particular choices of $B^n$ and $\mu$. See also Remark \ref{operads}.
\item If the stratification of $X$ is trivial then any smooth map $S^n \to X$ is transversal. Since $X$ is a manifold the Whitney approximation theorem allows us to approximate any continuous map and any homotopy by smooth ones, and thus $\psidot{n}{X} \cong \pi_n(X)$. 
\item Transversal maps $S^n \to X$ only `see' strata of codimension $n$ or less: $$\psidot{n}{X} \cong \psidot{n}{X_{\leq n}}$$ where $X_{\leq n}$ is the union of strata in $X$ of codimension $\leq n$.
\item For ease of reading we omit the basepoint $\star$ from the notation for the transversal homotopy monoids, but it is of course important. For two choices of basepoint in the same stratum the transversal homotopy monoids are isomorphic (since strata are connected). See \S \ref{first monoid} for how the first transversal homotopy monoid changes if we move the basepoint to a different stratum. 
\end{enumerate}
\end{remarks}

\begin{lemma}
\label{trivial elements}
An element $[f] \in \psidot{n}{X}$ where $n \geq 1$ is invertible if and only if the stratification induced by $f$ is trivial.
\end{lemma}
\begin{proof}
If the induced stratification of a transversal map $f:S^n \to X$ is not trivial then Proposition \ref{deformation} shows that the induced stratification of any other representative is non-trivial too. Hence the condition is invariant under homotopies through transversal maps. Furthermore if the stratification induced by $f$ is non-trivial then so is that induced by $(f\vee g)\circ \mu$  for any $g$ and so $[f]$ cannot be invertible. 

Conversely if the stratification induced by $f$ is trivial then $f$ maps $S^n$ into the open stratum containing the basepoint and the usual inverse of homotopy theory provides an inverse in $\psidot{n}{X}$.
\end{proof}
It follows that  $\psidot{n}{X}$ is not in general a group --- rather it is a dagger monoid. By this we mean an associative monoid, $M$ say, with anti-involution $a \mapsto a^\dagger$ such that
$1^\dagger=1$ and $(ab)^\dagger = b^\dagger a^\dagger$.  A \defn{map of dagger monoids} is a map $\varphi : M \to M'$ which preserves the unit, product and anti-involution.

The anti-involution on $\psidot{n}{X}$ is given by $[f]^\dagger = [f\circ \rho ]$ where $\rho : S^n \to S^n$ is the reflection in a hyperplane through the basepoint $\bpt$. Obviously there are many choices for $\rho$, but they all yield the same anti-involution. Note that when $n=0$ the anti-involution is trivial. 
\begin{remark}
\label{unitary isom}
Lemma \ref{trivial elements} implies that the dagger monoids which arise as transversal homotopy monoids have the special property that 
 $ab=1$ implies $b=a^\dagger$. Thus $a^\dagger$  is a `potential inverse' of $a$ or, put another way, all invertible elements are unitary. 
\end{remark}

By Lemma \ref{gen comp} the composition $s\circ f : S^n\to Y$ is transversal whenever $f:S^n\to X$ is transversal and $s:X \to Y$ is a stratified normal submersion. In this situation there is therefore a well-defined map $$\psidot{n}{s}:\psidot{n}{X} \to \psidot{n}{Y}.$$ For $n>0$ this is a map of dagger monoids. It is easy to complete the proof of the next result.
\begin{theorem}
There are functors $\psidot{n}{-}$ for $n>0$ from the category $\bsns$ of based Whitney stratified manifolds and stratified normal submersions to the category of dagger monoids. If $s$ and $s'$ are homotopic in $\bsns$ then $\psidot{n}{s}=\psidot{n}{s'}$, and consequently Whitney stratified manifolds which are homotopy equivalent in $\bsns$ have isomorphic transversal homotopy monoids. 
\end{theorem}

\subsection{An example: spheres}
\label{ex1}
We consider the transversal homotopy monoids of a sphere stratified by the antipode $p$ of the basepoint $\bpt$ and its complement $S^n-\{p\}$. We denote this space by $\mathbb{S}^k$. For $n,k>0$ the Pontrjagin--Thom construction, suitably interpreted, yields an isomorphism of dagger monoids
$$
	\psidot{n}{\mathbb{S}^k} \cong \frtang{}{k}{n}
$$
where $\frtang{}{k}{n}$ is the monoid of framed codimension $k$ tangles in $n$ dimensions. By this we mean it is the set of smooth ambient isotopy classes, relative to $B^n$, of  framed  codimension $k$ closed submanifolds of $S^n-B^n$. The monoidal structure is defined using the map $\mu:S^n \to S^n\vee S^n$ from the previous section --- given two submanifolds of $S^n-B^n$ consider their disjoint union as a submanifold of $S^n\vee S^n$ and take its pre-image under $\mu$. The unit is the empty submanifold and the dagger dual is obtained by applying the reflection $\rho:S^n \to S^n$. 

Here is the construction. There is a map $\iota: \psidot{n}{\mathbb{S}^k} \to \frtang{}{k}{n}$ defined by taking the induced stratification. More precisely, choose a framing of the point stratum $p\in \mathbb{S}^k$, \ie an orientation of $T_p\mathbb{S}^k$, and define $\iota[f]$ to be the (ambient isotopy class) of the pre-image $f^{-1}(p)$ with the pulled-back framing on 
$$
Nf^{-1}(p) \cong f^*T_p\mathbb{S}^k.
$$
To see that $\iota$ is well-defined we apply Remark \ref{smooth deformation} which yields the requisite smooth ambient isotopy of pre-images. 

An inverse to $\iota$ is provided by the `collapse map' construction. The proof of the following lemma is sketched in appendix \ref{collapse maps}.
\begin{lemma}
\label{collapse lemma}
Suppose $W$ is a smooth framed codimension $k$ closed submanifold of $S^n-B^n$. Then we can choose a \defn{collapse map} $\kappa_W:(S^n,B^n)\to (S^k,\bpt)$ for $W$ with the properties 
\begin{enumerate}
\item $\kappa_W^{-1}(p)=W$,
\item the restriction of $\kappa_W$ to $\kappa_W^{-1}(S^k-B^k)$ is a submersion,
\item the framing of $W$ agrees with that given by the isomorphism
$$
NW \cong \kappa^*_WT_pS^k \cong W \times \R^k.
$$
\end{enumerate}
The second property ensures that $\kappa_W: S^n\to \mathbb{S}^k$ is transversal.
If $W$ and $W'$ are ambiently isotopic (with the normal framings preserved) then any two choices of collapse maps $\kappa_W$ and $\kappa_W'$ are homotopic through transversal maps. Finally if $f:S^n\to \mathbb{S}^k$ is transversal then $f$ is homotopic to a collapse map for $f^{-1}p$ through transversal maps.  
\end{lemma}

\begin{corollary}
\label{IK isom}
There is a well-defined map $\kappa: \frtang{}{k}{n} \to \psidot{n}{\mathbb{S}^k}$ taking the class $[W]$ to the class $[\kappa_W]$. It is inverse to $\iota$.
\end{corollary}
\begin{proof}
The existence of $\kappa$ is immediate from the above lemma. The composite $\iota\kappa$ is the identity on representatives. By the last statement of the lemma the representative for $\kappa\iota[f]$ is homotopic to $f$ through transversal maps, so that $\kappa\iota[f]=[f]$. 
\end{proof}

To give some concrete examples, $\psidot{k}{\mathbb{S}^k}$ is the free dagger monoid on one generator when $k=1$ and the free {\em commutative} dagger monoid on one generator when $k\geq 2$. A more interesting example is provided by  $\psidot{3}{\mathbb{S}^2}$ which is the monoid of ambient isotopy classes of framed links. In particular, given any stratified normal submersion $s:\mathbb{S}^2 \to X$ the map 
$$
\psidot{3}{s}: \psidot{3}{\mathbb{S}^2} \to \psidot{3}{X}
$$
defines a framed link invariant valued in $ \psidot{3}{X}$. For example, the `forget the stratification' map $\mathbb{S}^2 \to S^2$ yields the self-linking number
$$
lk: \psidot{3}{\mathbb{S}^2} \to \pi_3(S^2) \cong \Z,
$$
\ie the linking number of a link $L$ with the link $L'$ obtained by moving $L$ off itself using the framing. 

\begin{remark}
Here is another way to obtain link invariants: given $L \subset S^3$ consider the transversal homotopy monoids of the space $S^3_L$ which is $S^3$ stratified by $L$ and $S^3-L$. Note that $\psidot{1}{S^3_L}$ is the fundamental group of the link complement, which is of course a well-known and powerful invariant. The higher invariants $\psidot{n}{S^3_L}$ seem novel.
\end{remark}

The above construction generalises to give an isomorphism 
\begin{equation}
\label{coloured IK isom}
\psidot{n}{\fatwedge_r \mathbb{S}^k} \cong \frtang{r}{k}{n}
\end{equation}
from the transversal homotopy monoid of a fat wedge $\fatwedge_r \mathbb{S}^k$ of spheres to the monoid of $r$-coloured framed codimension $k$ tangles in $n$ dimensions. Here `$r$-coloured' just means that each component of the tangle is labelled with one of $r$ `colours', and this labelling must be respected by the isotopies. The map is given by the induced stratification, with strata `coloured' by the point stratum of $\fatwedge_r\mathbb{S}^k$ to which they map. The proof that it is an isomorphism is similar to the uncoloured case, but now we use collapse maps $S^n \to \vee_r\mathbb{S}^k$ (or equivalently to the fat wedge). We omit the details. 

\begin{remark}
\label{operads}
One way to describe the structure of $n$-fold loop spaces is as algebras for the operad $U_n$ with
$$
U_n(r)= \cat{Top}_*(S^n, \vee_r S^n).
$$
The naive analogue for transversal homotopy theory fails because the composite of transversal maps need not be transversal. We can avoid this difficulty by working with collapse maps, which can be composed. There is an operad ${\rm Coll}_n$ with
$$
{\rm Coll}_n(r) = \{ \textrm{Collapse maps}\ S^n \to \vee_r\mathbb{S}^n \} \subset C^\infty(S^n,\vee_r S^n).
$$
Taking classes under homotopy through transversal maps, we see that $\psidot{n}{X}$ is an algebra for  an operad $\{ \psidot{n}{\fatwedge_r\mathbb{S}^n} \ | \ r\in \N \}$. For instance the product arises from $[\mu] \in \psidot{n}{\fatwedge_2\mathbb{S}^n}$, and any representative of this class gives the same product. Associativity of the product follows from the equation
$$
[ (1\vee \mu)\circ \mu] = [(\mu\vee 1)\circ \mu] \in \psidot{n}{\fatwedge_3\mathbb{S}^n} 
$$
and so on. Such equations can be visualised in terms of coloured isotopy classes. 
\end{remark}

\subsection{Products}
\label{products}

Homotopy groups respect products, that is $\pi_n(X \times Y) \cong \pi_n(X) \times \pi_n(Y)$. The situation is more complex for transversal homotopy monoids.

\begin{proposition}
Let $X$ and $Y$ be based Whitney stratified manifolds, and let ${\imath_X} :X \to X \times Y : x \mapsto (x,\bpt)$ be the inclusion and $\pi_Y: X \times Y \to Y$ the projection. Then there is a short exact sequence of dagger monoids
$$
1 \to \psidot{n}{X} \stackrel{\psidot{n}{\imath_X}}{\longrightarrow} \psidot{n}{X\times Y}\stackrel{\psidot{n}{\pi_Y}}{\longrightarrow}  \psidot{n}{Y} \to 1.
$$
Furthermore, the sequence is split in the sense that $\pi_X$ and $\imath_Y$ induce respectively a left and a right inverse for $\psidot{n}{\imath_X}$ and $\psidot{n}{\pi_Y}$. 
\end{proposition}
\begin{proof}
The proof is routine. To see that the sequence is exact in the middle, one applies Lemma \ref{trivial elements} to show that if $\psidot{n}{\pi_Y}[f] = [\pi_Y \circ f]$ is trivial then $\pi_Y \circ f$ factors through the open stratum $U \subset Y$ containing the basepoint. It follows that if $h$ is a homotopy in $Y$ through transversal maps from $\pi_Y\circ f$ to the constant map $q$ then $(\pi_X\circ f , h)$ is a homotopy in $X\times Y$ through transversal maps from $f$ to ${\imath_X} \circ \pi_X \circ f$.
\end{proof}

This proposition does not imply that $\psidot{n}{X\times Y} \cong \psidot{n}{X} \times \psidot{n}{Y}$. The simplest counterexample is when $X=Y=\mathbb{S}^1$ is a circle stratified by a point and its complement. In this case the short exact sequence for $n=1$ is
$$
1 \to \langle a \rangle \to \langle a,b \rangle \to \langle b \rangle \to 1
$$
where angled brackets denote the free dagger monoid on the specified generators (and the maps are the obvious ones). But it is certainly not the case that
$$
  \langle a,b \rangle \cong  \langle a \rangle \times  \langle b \rangle
  $$
as the latter is the free {\em commutative} dagger monoid on generators $a$ and $b$. Geometrically, the reason for this is that $\psidot{n}{-}$ depends only on the strata of codimension $\leq n$ and, in general, taking products introduces new strata of high codimension.
\begin{remark}
This illustrates a general problem in computing transversal homotopy monoids. The most often used tools for computing homotopy groups are the long exact sequence of a fibration and spectral sequences. Even if one had  analogues of these, they would be very weak tools in comparison because monoids do not form an abelian category. 
\end{remark}

\subsection{The first transversal homotopy monoid}
\label{first monoid}

We collect together a miscellany of observations and simple results about the structure of $\psidot{1}{X}$. We have already seen that there are restrictions on the monoids which can arise as $\psidot{1}{X}$ of some Whitney stratified manifold: they must be dagger monoids and all isomorphisms must be unitary (Remark \ref{unitary isom}). Proposition \ref{deformation} shows that the number of times a generic path in $X$ crosses a specified codimension $1$ stratum is an invariant of the path's class in  $\psidot{1}{X}$ which prohibits certain kinds of relations. It also prohibits the existence of an analogue of Eilenberg--MacLane spaces for transversal homotopy theory: if $\psidot{n}{X}$ contains an element with non-trivial induced stratification (\ie corresponding to a transversal map which meets higher codimension strata of $X$) then so does $\psidot{i}{X}$ for all $i>n$.

If the open (codimension $0$) strata of $X$ are simply-connected then $\psidot{1}{X}$ is the loop monoid of a finite directed graph with a source and target reversing involution on the edges.  Specifically, the graph has one vertex for each open stratum, a pair of edges in opposite directions between these vertices whenever the corresponding strata are separated by a codimension $1$ stratum with trivial normal bundle, and a loop at the ambient stratum for each codimension $1$ stratum with non-orientable normal bundle. The involution swaps pairs of edges corresponding to strata with trivial normal bundles and fixes the other edges. It follows directly from Proposition \ref{deformation} that an element of $\psidot{1}{X}$ is uniquely specified by the sequence of crossings (with orientation) of codimension $1$ strata. Thus $\psidot{1}{X}$ is isomorphic to the monoid of loops, based at the vertex corresponding to the stratum of the basepoint, in the directed graph. 

Conversely, any such graph can be realised starting from a closed stratified $3$-manifold. To construct the $3$-manifold, take a copy of $S^3$ for each vertex and connect sum whenever there is a pair $e$ and $e^\dagger$ of edges between two vertices. Whenever there is a loop $e=e^\dagger$ at a vertex, excise a disk from the corresponding $S^3$ and glue in the disk bundle of the canonical line bundle on $\R\P^2$. The result, after smoothing, is a closed $3$-manifold. Stratify it by taking a slice $S^2$ of each connect sum `bridge' and the zero section $\R\P^2$ of each added disk bundle as codimension $1$ strata. 

It is possible to give a similar combinatorial characterisation of those monoids which can arise as $\psidot{1}{X}$. In the general case one considers graphs with involution whose vertices are labelled by the fundamental groups of the corresponding open strata.

\begin{corollary}
If $X$ is a Whitney stratified manifold whose open strata are simply-connected then $\psidot{1}{X}$ is a quotient
$$
\langle l_i, i\in I \ | \ l_i =l_i^\dagger, i\in J \subset I \rangle
$$
of a free dagger monoid on generators $l_i$ for $i$ in a countable set $I$ subject to relations $l_i=l_i^\dagger$ for $i$ in a subset $J$. In particular $\psidot{1}{X}$ is free as a monoid (although not necessarily as a dagger monoid). 
\end{corollary}
\begin{proof}
Call a loop based at $v_0$ in a graph with involution {\em primitive} if it does not pass through $v_0$ except at the ends. If $l$ is primitive then so is $l^\dagger$. There are countably many primitive loops. Choose one from each set $\{ l, l^\dagger\}$ of primitive loops; the set of these choices is a generating set and the only relations are as stated.
\end{proof}
\noindent For an example in which $\psidot{1}{X}$ is not finitely generated consider $S^1$ with three point strata. 

Essentially the same argument as for the usual case gives the following `van Kampen' theorem. Alternatively, it can be deduced from the combinatorial description of $\psi_1$ in terms of graphs.
\begin{proposition}
\label{van kampen}
Let $X$ and $Y$ be Whitney stratified manifolds. Then $$\psidot{1}{X \fatwedge Y} \cong \psidot{1}{X} * \psidot{1}{Y}$$ is a free product.
\end{proposition}

\begin{proposition}
Let $X$ be a Whitney stratified manifold. Then the quotient of $\psidot{1}{X}$ obtained by adding  the relation $aa^\dagger =1$ for each $a\in \psidot{1}{X}$ is $\pi_1(X_{\leq 1})$ where $X_{\leq 1}$ is the union of strata of codimension $\leq 1$. 
\end{proposition}
\begin{proof}
As remarked earlier $\psidot{1}{X} \cong \psidot{1}{X_{\leq 1}}$. The map $\psidot{1}{X_{\leq 1}} \to \pi_1(X_{\leq 1})$ induced by forgetting the stratification is surjective because any loop is homotopic to a smooth transversal loop. It clearly factors through the quotient by the relations $aa^\dagger =1$. If smooth transversal loops are homotopic in $X_{\leq 1}$ then we can choose the homotopy $S^1 \times [0,1] \to X_{\leq1}$ to be smooth and transversal and such that the projection onto $[0,1]$ is Morse when restricted to the pre-image of the codimension $1$ strata. Such a homotopy can be decomposed as a composition of homotopies each of which is either a homotopy through transversal loops or (for each critical point) a homotopy corresponding to moving a bight of the loop over a stratum.

\vspace{.2in}
\centerline {
\includegraphics[width=1.25in]{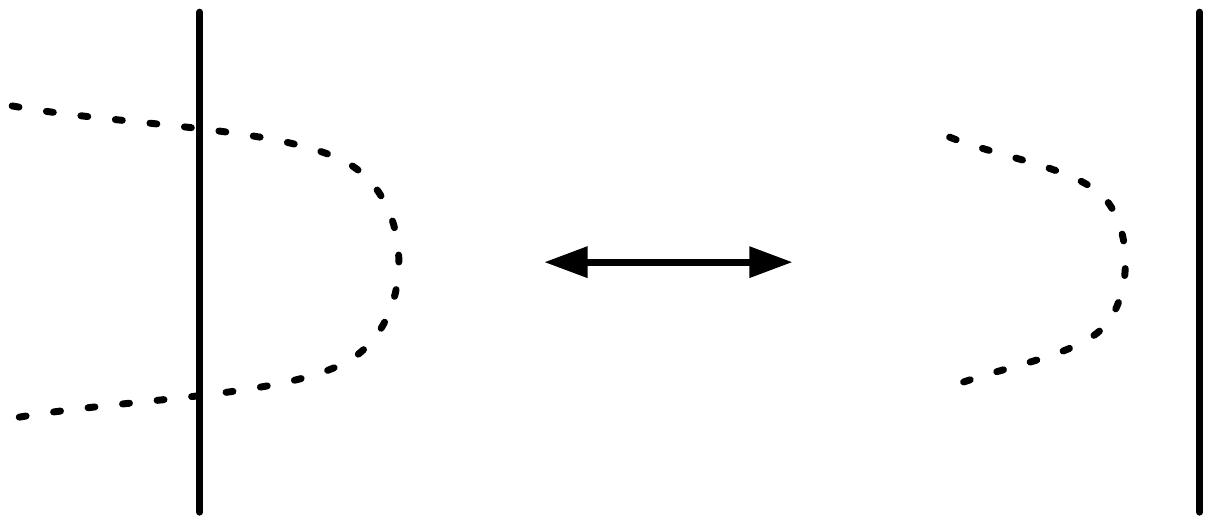}
}
\vspace{.2in}

\noindent These latter homotopies correspond to creating or cancelling a pair $aa^\dagger$. The result follows.
\end{proof}
There is no such relation between quotients of higher monoids obtained by turning duals into inverses and higher homotopy groups, essentially because there are other kinds of critical points in higher dimensions. For a concrete example, turning duals into inverses in $\psidot{2}{\mathbb{S}^1}$ does not yield $\pi_2(S^1)=0$ because the collapse map of a framed circle still represents a non-trivial class. 

\section{Transversal homotopy categories}
\label{transversal homotopy categories}

Let $X$ be a Whitney stratified manifold. As in the last section, all spaces are equipped with a generic basepoint $\bpt$ and all maps are based. Define the $n$th transversal homotopy category $\Psidot{n,n+1}{X}$ for $n\geq 0$ to be the category whose objects are transversal maps $(S^n,B^n) \to (X,\bpt)$. A morphism is represented by a transversal map $$f:(S^n \times [0,1], B^n \times [0,1]) \to (X,\bpt)$$ such that, for some $\epsilon >0$,  
\begin{equation}
\label{bdy cond}
f(p,t) = 
\left\{
\begin{array}{ll}
f(p,0) & t\in [0,\epsilon]\\
f(p,1) & t \in [1-\epsilon, 1].
\end{array}
\right.
\end{equation}
Two such maps represent the same morphism if they are homotopic through such maps relative to a neighbourhood of the boundary $S^n\times \{0,1\}$. Note that (\ref{bdy cond}) forces $f(-,0)$ and $f(-,1)$ to be transversal maps $S^n \to X$ and these are the respective source and target. Composition of morphisms is given by $[f] \circ [g] = [f \cdot g]$ where
$$
(f\cdot g)(p,t) = \left\{
\begin{array}{ll}
f(p,2t) & t\in[0,1/2)\\
g(p, 2t-1) & t \in [1/2,1]
\end{array}
\right.
$$
which is smooth because of the conditions (\ref{bdy cond}).

If $s:X \to Y$ is a stratified normal submersion then there is a functor 
$$
\Psidot{n,n+1}{s} : \Psidot{n,n+1}{X} \to \Psidot{n,n+1}{Y}
$$
given by post-composition, \ie on objects $f \mapsto s\circ f$ and on morphisms $[g] \mapsto [s \circ g]$. If $s$ and $s'$ are homotopic as maps in $\bsns$ then the corresponding functors are naturally isomorphic: $\Psidot{n,n+1}{s} \cong \Psidot{n,n+1}{s}$.

There are canonical equivalences of the categories defined with respect to different choices of the neighbourhood $B^n$. These equivalences are compatible with the functors induced by stratified normal submersions.

\subsection{An example: spheres again}
\label{ex2}

As in \S\ref{ex1} let  $\mathbb{S}^k$ denote the $k$-sphere stratified by a point $p$ and its complement, where $p$ is the antipode of the basepoint $\bpt$. Let 
$$\frtang{}{k}{n,n+1}$$ 
be the category of framed codimension $k$ tangles in dimensions $n$ and $n+1$. By this we mean the category whose objects are framed codimension $k$ closed submanifolds of $S^n -B^n$ and whose morphisms are ambient isotopy classes of framed codimension $k$ submanifolds of $(S^n-B^n) \times [0,1]$. The latter submanifolds are required to be of the form $M_0\times [0,\epsilon)$ and $M_1 \times (1-\epsilon,1]$ in neighbourhoods of $S^n \times \{0\}$ and $S^n\times \{1\}$ respectively. The framed submanifolds $M_0$ and $M_1$, which may be empty, represent the source and target respectively. The isotopies must  fix 
$$
\left(B^n \times [0,1]\right) \cup \left( S^n\times ([0,\epsilon) \cup (1-\epsilon,1]) \right).
$$
Composition is given by gluing cylinders $S^n \times [0,1]$ along their boundary components and re-parameterising.

Taking a transversal map $S^n \to \mathbb{S}^k$ to the corresponding induced stratification of $S^n$, and framing the codimension $k$ stratum which is the inverse image of $p$ by pulling back a framing of $p\in \mathbb{S}^k$, defines a functor  
$$
\iota : \Psidot{n,n+1}{\mathbb{S}^k} \to \frtang{}{k}{n,n+1}.
$$
It is well-defined on morphisms by Remark \ref{smooth deformation}. There is also a functor in the other direction
$$
\kappa :  \frtang{}{k}{n,n+1} \to \Psidot{n,n+1}{\mathbb{S}^k}
$$
given by choosing collapse maps for each framed submanifold and bordism. We can make these choices compatibly so that the collapse map for a bordism agrees with the chosen ones for the boundaries. Together $\kappa$ and $\iota$ define an equivalence: the existence of collapse maps for framed submanifolds shows that $\iota$ is essentially surjective --- indeed, surjective --- and a version of Corollary \ref{IK isom}, carried out relative to the faces $S^n\times \{0,1\}$, shows that is is fully faithful. 

For a concrete example, take $k=n=2$. The category $\frtang{}{2}{2,3}$ has finite collections of framed points in $S^2-B^2$ as objects, with ambient isotopy classes of framed curves in $(S^2-B^2)\times [0,1]$ possibly with boundary on the faces $S^2\times\{0,1\}$ as morphisms. 

There is also an equivalence $\Psidot{n,n+1}{\fatwedge_r \mathbb{S}^k} \simeq \frtang{r}{k}{n,n+1}$
from the transversal homotopy category of a fat wedge of spheres to the category of $r$-coloured framed codimension $k$ tangles in  dimensions $n$ and $n+1$, given by a  `coloured' version of the above argument.

\subsection{Structure of transversal homotopy categories}

Transversal homotopy categories have a rich structure, independent of the specific $X$. This structure is inherited from the transversal homotopy categories of spheres, and fat wedges of spheres. The idea is simple: given a suitable map $\alpha:S^n \to S^n$ we can define an endo-functor of $\Psidot{n,n+1}{X}$ by pre-composing with $\alpha$. The details, ensuring that all maps are transversal and so forth, are a little fiddly. For this reason we explain the construction for plain-vanilla homotopy theory and then state the conditions required for it to work in the transversal setting.

Let $(A,\bpt)$ be a CW complex with basepoint. Write $\Pi_{n,n+1}(A)$ for the category whose objects are based continuous maps $S^n\to A$ and whose morphisms are homotopy classes of maps
$$
(S^n\times[0,1], \{\bpt\}\times [0,1]) \to (A,\bpt)
$$
where the homotopies are relative to $S^n\times\{0,1\}$. The source and target are the restrictions to the slices $S^n\times\{0\}$ and $S^n\times\{1\}$ respectively, and composition is given by gluing cylinders. For example, in this notation, the fundamental groupoid is $\Pi_{0,1}(A)$.

A continuous map $\alpha:S^n \to \vee_r S^n$ determines a functor $$\alpha^* : \left(\Pi_{n,n+1}(A)\right)^k \to \Pi_{n,n+1}(A)$$ by pre-composition: on objects $\alpha^*(f_1,\ldots,f_r) = (f_1\vee \cdots\vee f_r)\circ \alpha$ and on morphisms
$$
\alpha^*([g_1], \ldots,[g_r]) = \left[ (g_1\vee \cdots\vee g_r)\circ (\alpha \times [0,1]) \right].
$$
(There is a mild abuse of notation here in which we write $g_1\vee \cdots \vee g_k$ for the map $\vee_r S^n\times [0,1]\to X$ defined by the $g_i$.) This definition is independent of the representatives $g_i$ chosen.

Similarly a continuous homotopy of based maps
$\beta: S^n \times [0,1] \to \vee_r S^n$
determines  a natural transformation $\beta^*$ from $\beta_0^*$ to $\beta_1^*$, where $\beta_t : S^n\times\{t\} \to S^n$ is the restriction to a slice. Namely, to each object $(f_1, \ldots, f_r)$ we associate the morphism $(f_1\vee \cdots\vee f_r) \circ \beta$. If $([g_1],\ldots,[g_r])$ is a morphism in $\Pi_{n,n+1}(A)^k$, then the composite
\begin{equation}
\label{naturality}
\xymatrix{
S^n\times[0,1]^2  \ar[rr]^{\beta \times [0,1]} &&  \vee_rS^n \times[0,1]\ar[rr]^{\qquad g_1\vee\cdots\vee g_r} && A
}
\end{equation}
provides a homotopy which shows that $\beta^*$ is a natural transformation. It depends only on the homotopy class of $\beta$ relative to the ends $S^n\times\{0,1\}$. Moreover, concatenating homotopies corresponds to composing natural transformations. We have proved
\begin{lemma}
Pre-composition defines a functor
$$
\Pi_{n,n+1}(\vee_r S^n) \to \left[ \left(\Pi_{n,n+1}(A)\right)^r, \Pi_{n,n+1}(A) \right]
$$
where $[\cat{C},\cat{D}]$ is the category of functors $\cat{C}\to \cat{D}$ and natural transformations between them. 
\end{lemma}

The majority of the construction carries over to the transversal homotopy setting provided we impose suitable conditions on $\alpha$ and $\beta$. The required conditions are that they should be smooth and that
\begin{enumerate}
\item the restriction of $\alpha$ to the inverse image of  $\sqcup_r(S^n -B^n)$ is a submersion,
\item the restriction of $\beta$, and of the slices $\beta_0$ and $\beta_1$, to the inverse image of  $\sqcup_r(S^n -B^n)$ are submersions.
\end{enumerate}
These ensure that the composites of $\alpha,\beta$ and the slices $\beta_0$ and $\beta_1$ with based transversal maps $\vee_rS^n \to X$ are transversal. There is one important difference however, which is that $\beta^*$ need no longer  be a {\em natural} transformation but merely a transformation of functors.\footnote{Here by a \defn{transformation} $t$ of functors $F,G:\cat{C}\to \cat{D}$ we mean simply a collection of morphisms $t(c):F(c) \to G(c)$ in $\cat{D}$ for each object $c$ of $\cat{C}$. What we call `transformations' are sometimes termed `infranatural transformations'.}  We denote the category of functors $\cat{C}\to\cat{D}$ and not-necessarily-natural transformations between them by $[\cat{C},\cat{D}]_+$.

To see why $\beta^*$ need not be natural, note that under the conditions above the map in (\ref{naturality}) is transversal, but is not necessarily a homotopy through transversal maps. However, if we impose the stronger condition that {\em each} slice $\beta_t$ restricts to a submersion on the inverse image of  $\sqcup_r(S^n -B^n)$ then $\beta^*$ is natural. For in this case the restrictions to slices
$$
\xymatrix{
S^n \times \{t\} \times [0,1]  \ar[rr]^{\beta_t \times [0,1]} &&  \vee_rS^n \times[0,1]\ar[rr]^{\qquad g_1\vee\cdots\vee g_r} && A
}
$$
in (\ref{naturality}) are transversal. Since transversality is an open condition we can find a family of transversal maps interpolating between the two ways around the boundary of  $S^n\times[0,1]$.
Indeed, under this stronger condition $\beta^*$ is a natural {\em isomorphism} because the morphism $[ (g_1\vee\cdots\vee g_r) \circ \beta]$ is a homotopy through transversal maps and therefore represents an isomorphism in $\Psidot{n,n+1}{X}$.

\begin{proposition}
\label{sub functors}
As in the previous section, let $\mathbb{S}^n$ denote the $n$-sphere stratified by a point and its complement. For each $r \geq 0$ there is a functor 
$$
 \kappa\iota(-)^* : \Psidot{n,n+1}{\fatwedge_r \mathbb{S}^n} \to \trans{ \left(\Psidot{n,n+1}{X}\right)^r }{ \Psidot{n,n+1}{X} }
$$
where $\iota$ and $\kappa$ are the functors defined in  \S\ref{ex2}.  Moreover, if $$h: S^n \times [0,1] \to \fatwedge_r \mathbb{S}^n$$ is a homotopy through transversal maps then $\kappa\iota(h)^*$ is a {\em natural isomorphism} of functors.
\end{proposition}
\begin{proof}
We use the functor $\kappa\iota$ to replace a transversal map $S^n \to\fatwedge_r \mathbb{S}^n$ with a collapse map 
$$
(S^n,B^n) \to (\vee_rS^n,\bpt)
$$
whose restriction to the inverse image of  $\sqcup_r(S^n -B^n)$ is a submersion. Then we apply the above construction of functors and not-necessarily-natural transformations.  

For the last part, if $h: S^n \times [0,1] \to \fatwedge_r \mathbb{S}^n$ is a homotopy through transversal maps then the collapse map $S^n\times [0,1] \to \vee_r S^n$ for the associated bordism has the property that each slice is a submersion onto $\sqcup_r(S^n -B^n)$. It follows that the corresponding transformation is a natural isomorphism.
\end{proof}

The force of this proposition is that it greatly simplifies the process of endowing transversal homotopy categories with structure. All that is required is to exhibit a few objects, morphisms and equations between morphisms in the transversal homotopy categories of fat wedges of spheres. We exhibit these diagrammatically as abstract versions of objects and morphisms in the coloured framed bordism category. Objects are represented by collections of coloured points, each with either a left or right pointing arrow indicating the framing (there are exactly two choices when $n> 0$), placed on a dotted line representing the ambient space. An empty dotted line represents the constant map to the basepoint. A single point with a right pointing arrow

\centerline {
\includegraphics[scale=0.5]{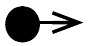}
}

\noindent represents the identity map $1: S^n\to S^n$, and with a left pointing arrow the reflection $\rho:S^n\to S^n$. Concatenations of these represent composites of $\mu : S^n \to S^n \vee S^n$, the bracketing is indicated by proximity. For example

\centerline {
\includegraphics[scale=0.5]{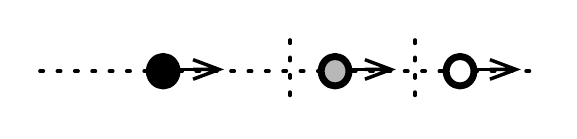}
}

\noindent represents $(1 \vee \mu)\circ \mu$. Morphisms are represented by coloured bordisms, equipped with arrows to indicate the framing. We read from top to bottom, so that, for example, 

\centerline {
\includegraphics[scale=0.5]{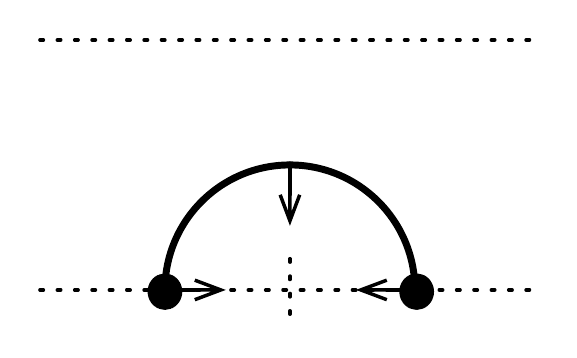}
}

\noindent represents a morphism $1 \to (1\vee \rho)\circ \mu$. Equations between morphisms are given by isotopies. Morphisms with no critical points for the horizontal projection represent homotopies through transversal maps.

\begin{theorem}
\label{structure thm}
For $n\geq1$ the transversal homotopy category $\Psidot{n,n+1}{X}$ is rigid monoidal, with the left and right duals given by the same functor. For $n\geq2$ there are braiding and balancing natural isomorphisms, giving $\Psidot{n,n+1}{X}$ the structure of a ribbon category. For $n\geq 3$ the braiding is symmetric and the balancing trivial. (See, for example, \cite{MR1797619} for definitions of rigid monoidal and ribbon categories etc.) 
\end{theorem}
\begin{proof}
The tensor product $\otimes$ and associativity natural isomorphism $\alpha$ are given respectively by the object and morphism below: 

\centerline {
\includegraphics[scale=0.5]{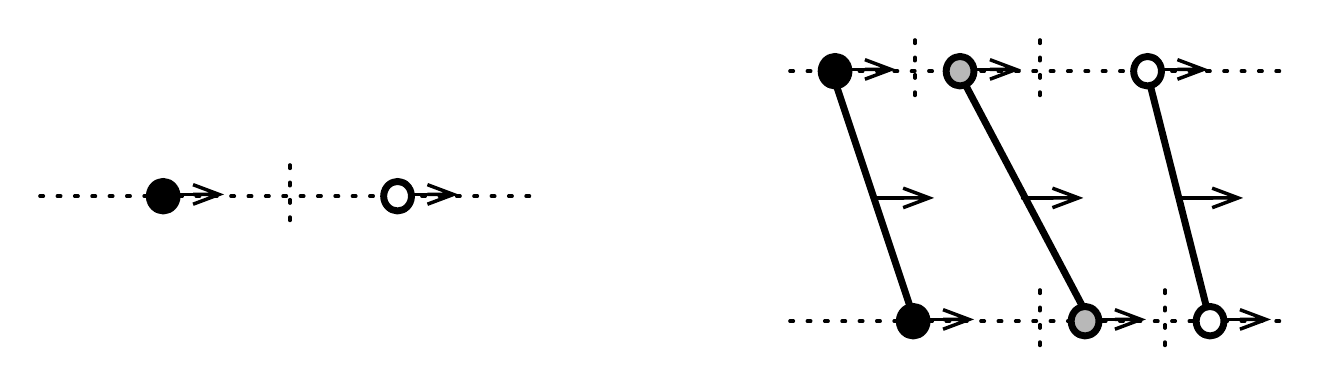}
}

\noindent The unit $1$ for the monoidal structure is the empty diagram with unit natural isomorphisms given by the following two morphisms

\centerline {
\includegraphics[scale=0.5]{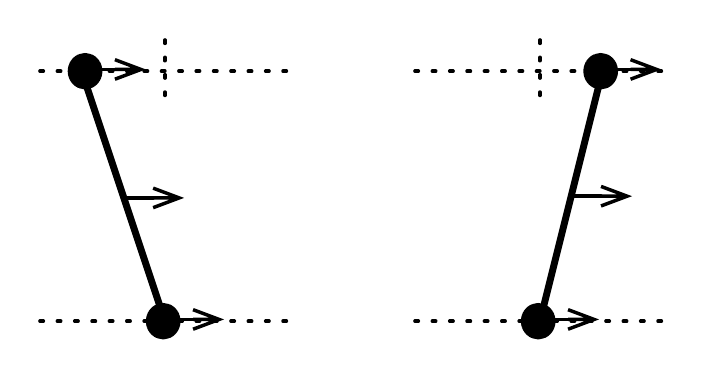}
}

\noindent To complete the monoidal structure two equations, the pentagon and triangle coherence axioms, must be satisified. These are easy to draw but  we omit them to save space. The (left and right) dual $f\mapsto f^\vee$ is given by the functor corresponding to 

\centerline {
\includegraphics[scale=0.5]{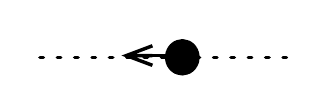}
}

\noindent The unit morphisms $\epsilon^L$ and $\epsilon^R$ for the left and right dual are shown respectively on the left and right below. Note that these morphisms correspond to {\em non-natural} transformations. 

\centerline {
\includegraphics[scale=0.5]{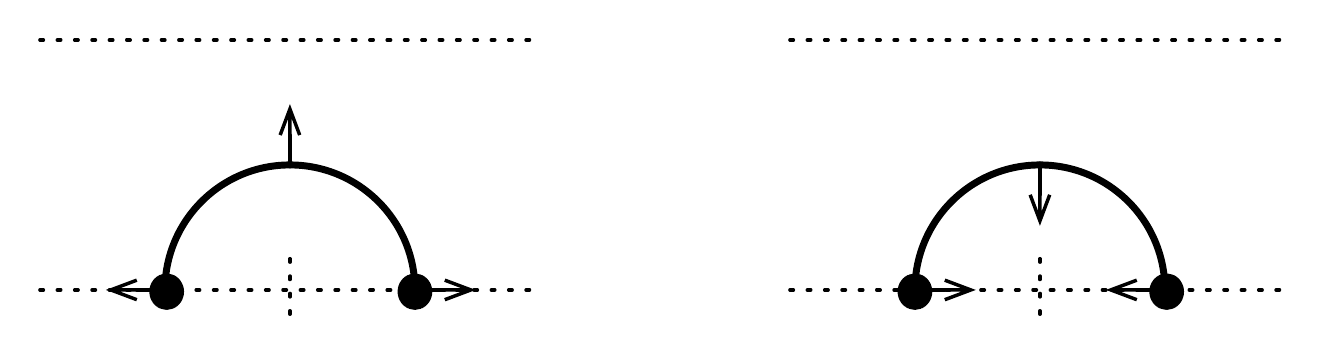}
}

\noindent The counit morphisms $\eta_L$ and $\eta^R$ for the left and right dual are given by:

\centerline {
\includegraphics[scale=0.5]{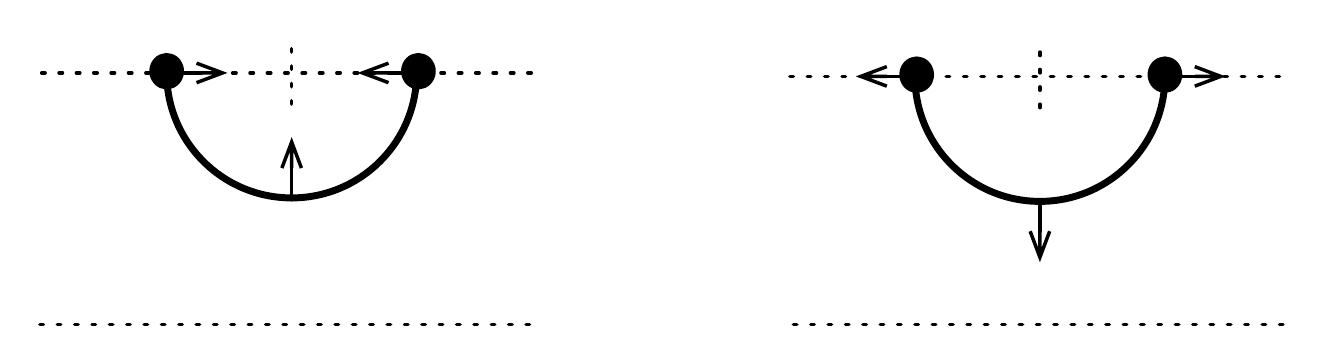}
}

\noindent To complete the proof that $\Psidot{n,n+1}{X}$ is rigid we need to show the rigidity axioms are satisified. One example is shown below --- the isomorphisms $1\otimes a \cong a$ and $a\otimes 1\cong a$ have been suppressed --- the other three are obtained by reflecting in the two axes.

\centerline {
\includegraphics[scale=0.5]{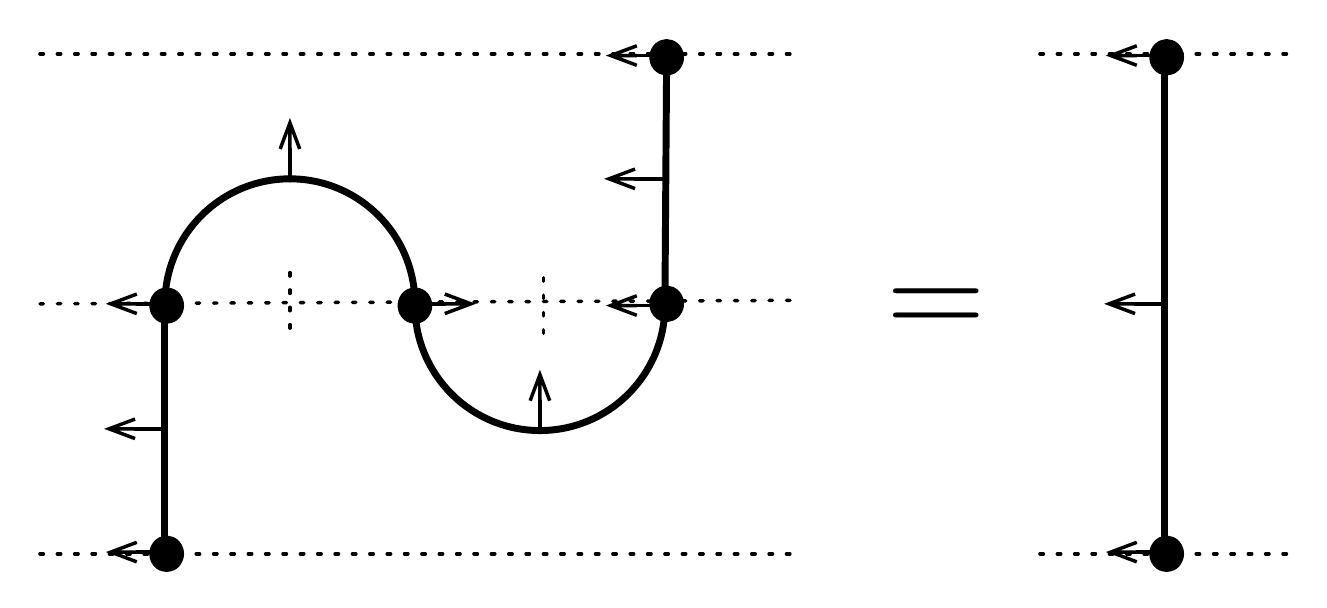}
}

When $n\geq 2$ there is a braiding natural isomorphism $\beta$ arising from the diagram below. The hexagon axioms relating this (and its inverse) to the associativity natural isomorphism are immediate.

\centerline {
\includegraphics[scale=0.5]{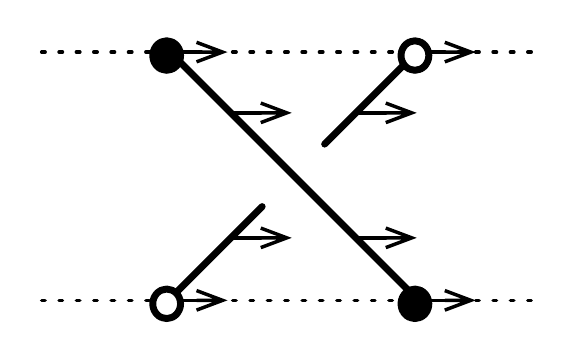}
}

\noindent Also when $n\geq 2$, there is a balancing natural isomorphism (or twist) $\tau$ corresponding to 

\centerline {
\includegraphics[scale=0.5]{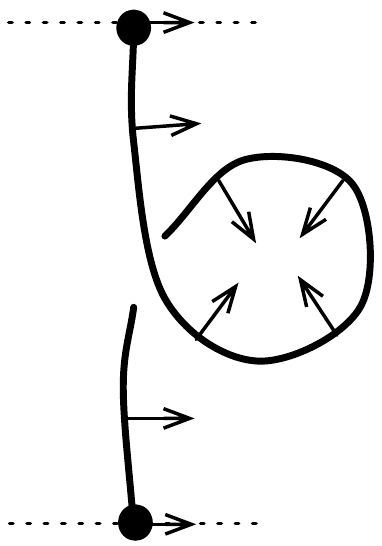}
}

\noindent satisfying the required balancing axioms. Finally, when $n \geq 3$ it is geometrically clear that the braiding isomorphism is an involution, \ie that the monoidal structure is symmetric, and that the balancing isomorphism is the identity. (Note that the double dual corresponds to $\rho\circ \rho = 1$ and so is the identity functor, and not merely naturally isomorphic to it.) 
\end{proof}
\begin{remark}
We made two choices in defining the above structure, the maps $\rho$ and $\mu$. Different choices lead to equivalent structures.
\end{remark}

There is one important piece of structure on transversal homotopy categories which does not arise in the above way, although it is still defined by pre-composition. It is the involutive anti-equivalence
$$
\dagger : \Psidot{n,n+1}{X}^{op} \longrightarrow \Psidot{n,n+1}{X}
$$
which is the identity on objects and is defined on morphisms by $[g]^\dagger= [g \circ \sigma]$ where 
$$
\sigma:S^n\times [0,1] \to S^n\times[0,1] : (p,t) \mapsto (p,1-t).
$$

\begin{proposition}
For $n,r\geq 0$ the involutive anti-equivalence $\dagger$ commutes with the functor $\kappa\iota(-)^*$. Furthermore
\begin{enumerate}
\item identities are unitary: $\id_f^\dagger = \id_f$,
\item units and counits for the left and right duals are adjoint: $(\epsilon_f^L)^\dagger = \eta_f^R$ and $(\epsilon_f^R)^\dagger = \eta_f^L$,
\item the braiding is unitary: $(\beta_{f\otimes g})^\dagger = \beta_{f\otimes g}^{-1}$,
\item and the balancing is unitary: $(\tau_f)^\dagger = \tau_f^{-1}$. 
\end{enumerate}
\end{proposition}
\begin{proof}
The fact that $\kappa\iota(-)^* \circ \dagger = \dagger \circ \kappa\iota(-)^*$ follows directly from the definitions. The other identities are consequences of this.
\end{proof}
\begin{corollary}
If $s:(S,\bpt) \to (Y,\bpt)$ is a stratified normal submersion of Whitney stratified manifolds then the functor
$$
\Psidot{n,n+1}{s} : \Psidot{n,n+1}{X} \to \Psidot{n,n+1}{Y} 
$$
preserves all of the structure defined above, \ie 
\begin{itemize}
\item it commutes with $\dagger$,
\item when $n\geq 1$ it is a (strict) monoidal functor commuting with the dual $\vee$ and preserving unit and counit morphisms,
\item when $n\geq 2$  it  preserves braiding and balancing isomorphisms,
\item and when $n\geq 3$ it is symmetric monoidal.
\end{itemize}
\end{corollary}
\begin{proof}
Composition on the left and right commute.
\end{proof}

\section{Thom spaces and stabilisation}
\label{thom}

The transversal homotopy theory of spheres has been our main example. The Pontrjagin--Thom construction relates it to the study of normally-framed submanifolds of $S^n$. Given that one has `fat Thom spaces' $\fatthom(E)$ in $\bsns$ for any vector bundle $E$ this generalises to any structure on the normal bundle. In other words there is an isomorphism
$$
\psidot{n}{\mathbb{M}G_k} \cong \tang{G}{}{k}{n},
$$
where $\mathbb{M}G_k$ is the fat Thom space of the bundle $EG_k \times_{G_k} \R^k \to BG_k$ and $\tang{G}{}{k}{n}$ is the monoid of $G$-tangles of codimension $k$ in dimension  $n$. (In other words it is the monoid  of ambient isotopy classes of codimension $k$ closed submanifolds of $S^n-B^n$ with a $G$-structure on the normal bundle.) Similarly, there is an equivalence
$$
\Psidot{n,n+1}{\mathbb{M}G_k} \cong \tang{G}{}{k}{n,n+1}.
$$
The special case when $G_k=1$ for all $k$ gives the earlier examples for spheres, since in this case $\mathbb{M}G_k \simeq \mathbb{S}^k$. Taking $G_k=SO(k)$ would give oriented isotopy and bordism and so on. To give a specific example, $\Psidot{2,3}{\mathbb{M}SO(2)}$ is (equivalent to) the category of oriented tangles. 

There is a `fat suspension' functor $\fatsusp : \bsns \to \bsns$ and one can check that 
$$
\psidot{n}{\fatsusp^rX} = 
\left\{
\begin{array}{ll}
0 & n<r \\
\pi_r(S^rU) & n=r
\end{array}
\right.
$$
where $U$ is the union of open strata in $X$. Suspension also defines maps (of dagger monoids) $$\psidot{n}{X} \to \psidot{n+1}{\fatsusp X}$$ so we can define the stable transversal homotopy monoids of $X$ to be
$$
\psi_{n}^{\rm S}(X) = \colim_{r\to\infty} \psidot{n+r}{\fatsusp^rX}.
$$
It is not immediately clear if there is an analogue of the Freudenthal suspension theorem in this setting. However, in the special case $X=\mathbb{S}^k$ the geometric interpretation of the transversal homotopy monoids suggests that
$$
\psi_{n}^{\rm S}(\mathbb{S}^k) \cong \colim_{r\to\infty} \psidot{n+r}{\mathbb{S}^{k+r}}
$$
will stabilise to the monoid of diffeomorphism classes of $(n-k)$-manifolds with a stable normal framing at precisely the expected point.

\section{The Tangle Hypothesis}
\label{tangle hypothesis}

When Baez outlined the idea of transversal homotopy theory in \cite{Baez:kx} one motivation was its relation to higher category theory and, in particular, the Tangle Hypothesis. We sketch this proposed relation, warning the reader that everything should be taken as provisional, since most of the objects and structures discussed are yet to be precisely defined.  

Let $\frtang{}{}{k,n+k}$ be the (conjectural) $k$-tuply monoidal $n$-category with duals  whose objects are $0$-dimensional framed submanifolds in $[0,1]^k$, morphisms are $1$-dimensional framed bordisms between such embedded in $[0,1]^{k+1}$, \ldots and whose $n$-morphisms are $n$-dimensional framed bordisms between bordisms between \ldots between bordisms embedded in $[0,1]^{n+k}$ considered up to isotopy. (For consistency with our earlier notation this would be 
$$
\frtang{}{k}{k,n+k}
$$
but, since the codimension is maximal, we drop it from the notation.) There is not yet a precise definition of `$k$-tuply monoidal $n$-category with duals' but see \cite{MR1355899} for the idea. Here are some examples of the structure for small $k$ and $n$ translated into our earlier language. (The need for a unifying terminology for the beasts in this menagerie is apparent!)
$$
\begin{array}{c | c | c | c}
 & n=0 & n=1 & \cdots \\ 
\hline
k=0 & \textrm{Set with involution} & \textrm{Dagger category} & \cdots \\
\hline
k=1 & \textrm{Dagger monoid} & \textrm{Rigid monoidal dagger}\\
&& \textrm{category} & \cdots \\
\hline
k=2 & \textrm{Commutative dagger} & \textrm{Ribbon dagger} &\\
& \textrm{ monoid} & \textrm{category} & \cdots \\
\hline
k=3 & \vdots & \textrm{Rigid commutative monoidal} &\\
&& \textrm{dagger category} &  \cdots \\
\hline
k=4 & \vdots & \vdots & \ddots
\end{array}
$$
The vertical dots indicate that the remaining entries in the column are expected to be the same, \ie that the structure stabilises with increasing $k$. This table is a `with duals' version of the `periodic table of higher category theory' see \cite{MR1355899}.

The Tangle Hypothesis \cite{MR1355899} proposes an algebraic description of $\frtang{}{}{k,n+k}$. Specifically it suggests that it is equivalent (in an appropriate sense) to the free $k$-tuply monoidal $n$-category with duals on one object. See \cite{MR1268782} for a proof in the case when $k=2,n=1$ and \cite{MR2020556} for other references, known cases and related results. See \cite{Lurie:2009fk} for a sketch proof of the Cobordism Hypothesis, the stable version of the Tangle Hypothesis. 

Aside from the elegance of the statement, the importance of the Tangle Hypothesis is that, if true, it provides a standard procedure for defining invariants of framed tangles (and thereby of manifolds, bordisms etc). Suppose $\cat{C}$ is some interesting $k$-tuply monoidal $n$-category with duals. If we fix an object of $\cat{C}$ then the free-ness of $\frtang{}{}{k,n+k}$ guarantees a unique structure-preserving $n$-functor $\frtang{}{}{k,n+k} \to \cat{C}$. The values of this functor, or related quantities, are the invariants. The prototype for this strategy is the Jones polynomial of a link.

Transversal homotopy theory enters the story because one expects to be able to define a $k$-tuply monoidal $n$-category with duals $\Psidot{k,k+n}{X}$ for each Whitney stratified manifold $X$. Furthermore, one expects an equivalence (of $k$-tuply monoidal $n$-categories with duals)
\begin{equation}
\label{tangle eq}
\frtang{}{}{k,n+k} \simeq \Psidot{k,k+n}{\mathbb{S}^k}
\end{equation}
generalising the earlier examples  in \S \ref{ex1} and \S\ref{ex2}.

The proposed definition of $\Psidot{k,n+k}{X}$ is a straightforward generalisation of our earlier definitions. For $0 \leq i < n$ an $i$-morphism in $\Psidot{0,n}{X}$ is a transversal map 
$$
f:[0,1]^i \to X
$$
and an $n$-morphism is an equivalence  class of transversal maps $[0,1]^n\to X$ under the relation generated by homotopy through transversal maps. This ensures that there is a well-defined associative composition for $n$-morphisms given by juxtaposition in the $n$th coordinate direction. Composition of $i$-morphisms for $0< i <n$ is defined analogously by juxtaposition in the $i$th coordinate direction. This will only be associative and unital up to higher morphisms (as expected in an $n$-category). To obtain $\Psidot{k,n+k}{X}$ we choose a generic basepoint $\bpt\in X$ and take the full subcategory of $\Psidot{0,n+k}{X}$ where $i$-morphisms for $0\leq i < k$ are the constant map to $\bpt$. 

When $n=0$ we have $\Psidot{k,k}{X}=\psidot{k}{X}$ and when $n=1$ the above agrees with the definition of transversal homotopy category in \S\ref{transversal homotopy categories}. We've already seen that these have the expected structures of $k$-tuply monoidal $n$-categories with duals, so things work nicely for $n\leq 1$.

Taking the Tangle Hypothesis and (\ref{tangle eq}) on trust, every stratified normal submersion $\mathbb{S}^n \to X$ yields a framed tangle invariant valued in $\Psidot{k,n+k}{X}$. More ambitiously, the equivalence (\ref{tangle eq}) may throw light on the Tangle Hypothesis itself. Certainly $\Psidot{k,n+k}{\mathbb{S}^k}$ should have an appropriate universal property for $k$-tuply monoidal $n$-categories of the form $\Psidot{k,n+k}{X}$.

\appendix

\section{Collapse maps}
\label{collapse maps}

We explain how to construct collapse maps for framed submanifolds satisfying the properties laid out in Lemma \ref{collapse lemma}. The construction is the standard one with a few refinements to obtain the required properties for our purposes. A good general reference for the Pontrjagin--Thom construction is \cite[Chapter 7]{MR1487640}.

 Let $\sigma : \mathbb{S}^k -\{\bpt\} \to \R^k$ be stereographic projection. As usual, for each $n$, we fix a small closed disk-shaped neighbourhood $B^n$ of the basepoint in $S^n$. Suppose  $W$ is a framed codimension $k$ closed submanifold in $S^n-B^n$. Choose a tubular neighbourhood $U$ of $W$ in $S^n-B^n$, a diffeomorphism $\tau : U \cong NW$ and a bundle isomorphism $\phi : NW \cong W \times \R^k$ representing the framing. Define the collapse map
$$
\kappa_W : (S^n,B^n) \to (\mathbb{S}^k,\bpt)
$$
to be a smoothing of the continuous map which is constant with value $\bpt$ on $S^n - U$ and the composite
$$
U \stackrel{\tau}{\longrightarrow} NW  \stackrel{\phi}{\longrightarrow} W \times \R^k \stackrel{\pi}{\longrightarrow} \R^k  \stackrel{\sigma^{-1}}{\longrightarrow} \mathbb{S}^k
$$
on $U$. We can choose this smoothing relative to $S^n-U$ and a closed disk-bundle neighbourhood of $W$ within $U$. It follows that $\kappa_W^{-1}(p)=W$ and $\kappa_W^{-1}(\bpt)\supset B^n$ and that the framing is given by the isomorphism $NW\cong \kappa_W^*T_pS^k$. Furthermore, since submersions are stable, we can choose the smoothing sufficiently small that that the restriction of $\kappa_W$ to the inverse image of $S^k-B^k$ remains submersive. Thus the properties in lemma \ref{collapse lemma} hold.

The construction of $\kappa_W$ depends on choices of tubular neighbourhood, the diffeomorphisms $\tau$ and $\phi$ and the smoothing. Suppose we make two different sets of choices leading to two different collapse maps $\kappa_W$ and $\kappa_W'$. One can construct a collapse map for $W \times [0,1] \subset S^n \times [0,1]$ (with the obvious framing induced from that of $W$) by making choices which agree with the two given ones on the ends $W \times \{0,1\}$. The details, which follow from the essential uniqueness of tubular neighbourhoods and so on, are left to the reader. The resulting collapse map provides a homotopy through transversal maps from $\kappa_W$ to $\kappa_W'$.

Now suppose $W$ and $W'$ are framed submanifolds of $S^n-B^n$ representing the same class in $\frtang{}{k}{n}$. Thus there is a smooth ambient isotopy $\alpha : S^n \times [0,1]\to S^n$ taking $W$ to $W'$ and  fixing $B^n$. The composite
$$
S^n\times [0,1] \stackrel{\alpha}{\longrightarrow} S^n \stackrel{\kappa_W}{\longrightarrow} \mathbb{S}^k
$$
is a homotopy through transversal maps from $\kappa_W$ to a collapse map for $W'$. This establishes the desired uniqueness of collapse maps up to homotopy through transversal maps. 

Finally we show that if $f:(S^n,B^n)\to (\mathbb{S}^k,\bpt)$ is a transversal map then $f$ is homotopic through transversal maps to some collapse map for the framed submanifold $f^{-1}(p)$. To see this let $B_\epsilon(0)$ be the $\epsilon$-ball about the origin in $\R^k$ and let  $U_\epsilon = \sigma^{-1}B_\epsilon(0)$. For sufficiently small $\epsilon >0$, we can choose a diffeomorphism $\varphi$ so that
$$
\xymatrix{
f^{-1}(p) \times U_\epsilon\ar[r]^\varphi \ar[dr]_\pi &  f^{-1}U_\epsilon\ar[d]^f\\
& U_\epsilon
}
$$
commutes where $\pi$ is the projection. It follows that $f$ is smoothly homotopic to a collapse map for the submanifold $f^{-1}(p)$ via
$$
S^n \times [0,1] \stackrel{f \times \id}{\longrightarrow} \mathbb{S}^k \times [0,1] \stackrel{\eta}{\longrightarrow} \mathbb{S}^k
$$
where $\eta$ is a smooth homotopy from the identity to a map which fixes $p$ and contracts $\mathbb{S}^k - U_\epsilon$ to $\bpt$. Furthermore we can ensure that each slice of this homotopy is transversal by insisting that the homotopy fixes a neighbourhood of $p$ point-wise.

\bibliographystyle{plain}
\bibliography{homotopy}
\end{document}